\documentclass[12pt,a4paper,twoside,reqno]{amsart}
\title[{$\mathrm{G}_2$-structures and the deformed Shatashvili--Vafa vertex algebra}]{$\mathrm{G}_2$-structures with torsion and the deformed Shatashvili--Vafa vertex algebra}

\author[A. De Arriba de La Hera]{Andoni De Arriba de La Hera}
\address{Universidad Nacional de Educaci\'on a Distancia (UNED)\\ C. de Juan del Rosal, 10, Moncloa-Aravaca,\\ 28040 Madrid, Spain}
\email{andoni.dearriba@mat.uned.es}

\author[M. Galdeano]{Mateo Galdeano}
\address{Fachbereich Mathematik, Universit\"at Hamburg\\ Bundesstrasse 55\\ 20146, Hamburg, Germany}
\email{mateo.galdeano@uni-hamburg.de}

\author[M. Garcia-Fernandez]{Mario Garcia-Fernandez}
\address{Instituto de Ciencias Matem\'aticas (CSIC-UAM-UC3M-UCM)\\ Nicol\'as Cabrera 13--15, Cantoblanco\\ 28049 Madrid, Spain}
\email{mario.garcia@icmat.es}

\thanks{
Partially supported by the Spanish Ministry of Science and Innovation, through the `Severo Ochoa Programme for Centres of Excellence in R\&D' (CEX2023-001347-S). The first author's work is funded by the Deutsche Forschungsgemeinschaft (DFG, German Research Foundation) under Germany's Excellence Strategy, EXC 2121 ``Quantum Universe,'' 390833306. The second and third authors are partially supported by MICINN under grants PID2019-109339GA-C32, PID2022-141387NB-C22 and CNS2022-135784. 
}

\usepackage{hyperref,url}
\usepackage{cleveref} 
\usepackage{mathtools} 
\usepackage{amssymb,latexsym}
\usepackage{amsmath,amsthm}
\usepackage[latin1]{inputenc}
\usepackage{enumerate}
\usepackage[all]{xy} \CompileMatrices
\SelectTips{cm}{12} 
\usepackage{tikz-cd}

\usepackage{verbatim} 
\usepackage{wrapfig}
\usepackage{graphicx}
\usepackage{slashed}
\usepackage{multicol}

\theoremstyle{plain}
\newtheorem{theorem}{Theorem}[section]
\newtheorem{lemma}[theorem]{Lemma}

\newtheorem{proposition}[theorem]{Proposition}
\newtheorem{conjecture}[theorem]{Conjecture}
\newtheorem*{theorem*}{Theorem}
\theoremstyle{definition}
\newtheorem{definition}[theorem]{Definition}
\newtheorem{definition-theorem}[theorem]{Definition-Theorem}
\newtheorem{example}[theorem]{Example}
\newtheorem*{acknowledgements}{Acknowledgements}
\theoremstyle{remark}
\newtheorem{remark}[theorem]{Remark}

\numberwithin{equation}{section} 

\newcommand{\I}{{\mathbb{I}}}

\newcommand{\dd}{{\rm d}}

\newcommand{\SV}{\mathrm{SV}}
\newcommand{\normord}[1]{:\mathrel{#1}:} 

\allowdisplaybreaks

\begin{document}

\begin{abstract}
We construct representations of the deformed Shatashvili--Vafa vertex algebra $\SV_a$, with parameter $a \in \mathbb{C}$, as recently proposed in the physics literature by Fiset and Gaberdiel. The geometric input for our construction are integrable $\mathrm{G}_2$-structures with closed torsion, solving the heterotic $\mathrm{G}_2$ system with $\alpha'=0$ on the group manifolds $S^3\times T^4$ and $S^3\times S^3\times S^1$. From considerations in string theory, one expects the chiral algebra of these backgrounds to include $\SV_a$, and we provide a mathematical realization of this expectation by obtaining embeddings of $\SV_a$ in the corresponding superaffine vertex algebra and the chiral de Rham complex. In our examples, the parameter $a$ is proportional to the scalar torsion class of the $\mathrm{G}_2$ structure, $a \sim \tau_0$, as expected from previous work in the semi-classical limit by the second author, jointly with de la Ossa and Marchetto.
\end{abstract}

\maketitle

\setlength{\parskip}{5pt}
\setlength{\parindent}{0pt}


\section{Introduction}\label{sec:intro}

A supersymmetric background of string theory is described by a superconformal field theory on the worldsheet, and the underlying superconformal algebra of chiral symmetries is of special interest. Mathematically, this chiral algebra is identified with a supersymmetric (SUSY) vertex algebra, and the correspondence with the background geometry can be made precise by providing an embedding of the vertex algebra in the chiral de Rham complex.

The chiral de Rham complex is a sheaf of vertex algebras on any smooth manifold, introduced by Malikov, Schechtman and Vaintrob \cite{Malikov:1998dw}, which provides a formal quantization of the non-linear sigma model \cite{Ekstrand:2009zd}. Vertex algebra embeddings have been previously studied for special holonomy manifolds, see for example \cite{Ekstrand:2010wu,Heluani:2017juq,RodriguezDiaz:2016tih}.

Embeddings of the $\mathcal{N}=2$ vertex algebra for compact non-K\"ahler complex manifolds have been recently obtained, first in the homogeneous setting \cite{Alvarez-Consul:2020hbl} and later in more generality from solutions of the Hull-Strominger system \cite{Alvarez-Consul:2023zon}. The key assumption in these results is that the underlying manifold admits a solution to the \emph{Killing spinor equations} of heterotic supergravity. In particular, these geometries describe heterotic supersymmetric backgrounds to leading order in the $\alpha'$-expansion.

Our main motivation in this note is to study the analogous problem for a seven-dimensional Riemannian manifold $M$. Suppose $M$ admits a solution to the Killing spinor equations with parameter $\lambda\in\mathbb{R}$ and closed NS flux $H$ \cite{Silva:2024fvl}:
\begin{equation}
\label{eq:Killingspinoreqs}
    \nabla^+\eta=0 \, , \qquad \left( \slashed{\nabla}^{1/3}-\frac{1}{2}\zeta \right) \cdot \eta = \lambda\, \eta \, , \qquad \dd H=0 \, ,
\end{equation}
for a real spinor $\eta$, a three-form $H\in\Omega^3$, and a one-form $\zeta\in\Omega^1$. Here, $\nabla^+$ and $\slashed{\nabla}^{1/3}$ are the spin connection and Dirac operator of the connections with skew torsion $H$ and $\frac{1}{3} H$, respectively. The equations \eqref{eq:Killingspinoreqs} are equivalent \cite{Silva:2024fvl} to the \emph{heterotic $\mathrm{G}_2$ system} with $\alpha'=0$ \cite{Clarke:2016qtg,delaOssa:2017pqy}. Consequently, there exists an integrable $\mathrm{G}_2$-structure on $M$ with torsion three-form $H$ (see Proposition \ref{prop:uniqueG2connection}) satisfying
\begin{equation*}
    \dd\varphi = \frac{12}{7}\lambda\, \psi + 3 \, \tau_1\wedge\varphi + *\tau_3 \, , \qquad
    \dd\psi  = 4 \, \tau_1\wedge\psi \, , \qquad \dd H=0 \, ,
    \end{equation*}
where the precise definition of the different torsion classes $\tau_j$ can be found in \Cref{sec:G2structures}. Note that, in the present setup, a solution to the first and third equations in \eqref{eq:Killingspinoreqs} determines uniquely a solution of the full system, with the prescription 
$$
\tau_0=\frac{12}{7}\lambda, \qquad \tau_1=\frac{1}{4}\zeta.
$$
When $\tau_1$ is exact, the system describes $\mathcal{N}=1$ backgrounds of the form $\mathrm{AdS}_3\times M$ when $\tau_0\neq0$, and $\mathbb{R}^{1,2}\times M$ when $\tau_0=0$. Remarkably, the eigenvalue $\lambda$ of the Dirac spinor $\eta$ is tied to the scalar torsion class $\tau_0=\frac{12}{7}\lambda$, which is in turn related to the cosmological constant of the spacetime.

In the case of a $\mathrm{G}_2$-holonomy manifold, all torsion classes vanish and the spacetime degenerates to flat Minkowski space. In this situation, the underlying SUSY vertex algebra was described by Shatashvili--Vafa \cite{Shatashvili:1994zw}, and its embedding in the chiral de Rham complex was found by Rodr\'iguez D\'iaz \cite{RodriguezDiaz:2016tih}.

Fiset--Gaberdiel \cite{Fiset:2021azq} proposed a one-parameter family of algebras, that we call \emph{deformed Shatashvili--Vafa} and denote by $\SV_a$, which depend on a parameter $a \in \mathbb{C}$ and are suitable for describing $\mathrm{AdS}_3$ backgrounds. These algebras belong to the two-parameter family of algebras $\mathcal{SW}(\frac{3}{2},\frac{3}{2},2)$ first constructed by Blumenhagen \cite{Blumenhagen:1991nm} (see Remark \ref{rem:SW} for a more precise statement). By studying the semiclassical limit---where the superconformal field theory induces a Poisson Vertex algebra---it was shown in \cite{delaOssa:2024cgo} that the parameter of $\SV_a$ is tied to the scalar torsion class, $a=i\frac{7}{6}\tau_0\ell_s$, where $\ell_s$ is the string length.

Gathering all the evidence above, we propose the following conjecture:

\begin{conjecture}
    Let $M$ be a seven-dimensional Riemannian manifold admitting a solution to \eqref{eq:Killingspinoreqs}. Then, its chiral de Rham complex admits an embedding of the SUSY vertex algebra $\SV_a$, where the value of $a$ is determined by the eigenvalue $\lambda$ of the Dirac spinor $\eta$.
\end{conjecture}

Similarly as in \cite{Alvarez-Consul:2023zon}, we also expect an analogous result to hold for the more general situation where a gauge bundle equipped with a $\mathrm{G}_2$-instanton solving the heterotic Bianchi identity is incorporated. In this case, one has a Killing spinor on a string algebroid describing a solution of the heterotic $\mathrm{G}_2$ system with $\alpha'\neq 0$, and the chiral de Rham complex should be the one associated to this algebroid.

In this note, we provide further evidence for the conjecture by showing that it holds for five different families of $\mathrm{G}_2$-structures, defined over two different homogeneous manifolds following \cite{Alvarez-Consul:2020hbl,Fino:2023vdp}. Our main result, including the relationship between $a$ and $\tau_0$, is presented in \Cref{thm:finalresult} below.

It is worth pointing out that the torsion class $\tau_3$ does not play a role in our construction. On the other hand, the torsion class $\tau_1$, which is always closed in the examples we consider, appears in an analogous fashion to that of the Lee form in the $\mathcal{N}=2$ case studied in \cite{Alvarez-Consul:2020hbl, Alvarez-Consul:2023zon}, see \Cref{rem:tau1enG}. Similarly to the results of \cite{Alvarez-Consul:2020hbl, Alvarez-Consul:2023zon}, we also expect that a dilaton-corrected version of \Cref{thm:finalresult} should hold.

\begin{acknowledgements}
We would like to thank Andrew Linshaw for help with the program \emph{OPEdefs}, and Jethro van Ekeren for helpful discussions. The first author is grateful to the Department of Mathematics of the University of Denver for the hospitality. The second and third authors would like to thank the mathematical research institute MATRIX in Australia for hospitality and support during the program \emph{The Geometry of Moduli Spaces in String Theory}, where some of the ideas of this note were born. We would specially like to thank the organizers of the program for providing a stimulating environment for discussions.
\end{acknowledgements}

\section{Integrable \texorpdfstring{$\mathrm{G}_2$}{G2}-structures and their torsion classes}
\label{sec:G2structures}

\subsection{Torsion of a \texorpdfstring{$\mathrm{G}_2$}{G2}-structure}

We begin with a brief summary of $\mathrm{G}_2$-structures and introduce the properties which will be relevant to us. For more detailed accounts on the topic and proofs of the statements presented, we refer the reader to \cite{Bryant:2005mz, Karigiannis:2020}.

\begin{definition}
    Let $M$ be a seven-dimensional manifold. A \emph{$\mathrm{G}_2$-structure} on $M$ is a nowhere-vanishing three-form $\varphi$ on $M$ such that for every $p\in M$ there exists a basis $\lbrace v^1,\dots,v^7\rbrace$ of $T_p^*M$ for which
\begin{equation*}
\label{eq:varphi0}
\varphi_0=v^{123}-v^{145}-v^{167}-v^{246}+v^{257}-v^{347}-v^{356} \, ,
\end{equation*}
where we are writing $v^{ij}=v^i\wedge v^j \, $. We call $\varphi$ the \emph{associative form}.
\end{definition}

A $\mathrm{G}_2$-structure determines an orientation and a metric $g$ on $M$, given by
    \begin{equation*}
        g(X,Y)\,\dd vol=\frac{1}{6}\left(X\lrcorner\varphi\right)\wedge\left(Y\lrcorner\varphi\right)\wedge\varphi \, ,
    \end{equation*}
for any $X,Y\in TM$. We can use these to define the \emph{coassociative} form $\psi= * \varphi$.

Differential forms on $M$ decompose into irreducible $\mathrm{G}_2$-representations, inducing a decomposition of the spaces $\Omega^k=\Gamma\left(\Lambda^k(T^*M)\right)$. Two of these subspaces which will shortly play an important role are
\begin{equation*}
        \Omega^2_{14}=\lbrace \beta\in\Omega^2 \, \vert \, \beta\wedge\psi=0 \rbrace \, , \qquad \Omega^3_{27}=\lbrace \gamma\in\Omega^3 \, \vert \, \gamma\wedge\varphi=0 \, , \, \gamma\wedge\psi=0 \rbrace \, .
\end{equation*}
We can uniquely characterize the torsion of the $\mathrm{G}_2$-structure by decomposing the exterior derivatives of $\varphi$ and $\psi$ into irreducible $\mathrm{G}_2$-representations:

\begin{definition}
     Let $M$ be a seven-dimensional manifold and let $\varphi$ be a $\mathrm{G}_2$-structure on $M$. The \emph{torsion classes} of the $\mathrm{G}_2$-structure are the differential forms $\tau_0 \in \Omega^{0}$, $\tau_1 \in \Omega^{1}$, $\tau_2 \in \Omega^{2}_{14}$ and $\tau_3 \in \Omega^{3}_{27}$ uniquely determined by
     \begin{equation*}
    \dd\varphi = \tau_0\, \psi + 3 \, \tau_1\wedge\varphi + *\tau_3 \, , \qquad
    \dd\psi  = 4 \, \tau_1\wedge\psi + *\tau _2 \, .
    \end{equation*}
    We will say that the $\mathrm{G}_2$-structure is \emph{integrable} if $\tau_2=0$.
\end{definition}

\begin{proposition}[\cite{Friedrich:2001nh}]
\label{prop:uniqueG2connection}
    Let $M$ be a seven-dimensional manifold and let $\varphi$ be an integrable $\mathrm{G}_2$-structure on $M$. Then, there exists a unique connection with totally skew-symmetric torsion compatible with the $\mathrm{G}_2$-structure. Its torsion three-form $H$ is given by
    \begin{equation*}
    H =  \frac{1}{6}\, \tau_0\, \varphi - \tau_1\lrcorner\,\psi - \tau_3 \, .
    \end{equation*}
\end{proposition}

\subsection{\texorpdfstring{$\mathrm{G}_2$}{G2}-structures via products}

A simple way of constructing a $\mathrm{G}_2$-structure is via $\mathrm{G}$-structures in lower-dimensional manifolds. To this end, we introduce:

\begin{definition}
    Let $X$ be a four-dimensional manifold. An \emph{$\mathrm{SU}(2)$-structure} on $X$ is a triple $\lbrace\omega_i\rbrace_{i=1,2,3}$ of nowhere-vanishing two-forms on $X$ satisfying
\begin{equation*}
\label{eq:SU2relations}
\omega_i \wedge\omega_j = 0 \,, \qquad \frac{1}{2}\, \omega_i\wedge\omega_i = \frac{1}{2}\, \omega_j\wedge\omega_j \neq 0 \, ,
\end{equation*}
for each $i,j\in\lbrace1,2,3\rbrace$, $i\neq j$.
\end{definition}

\begin{definition}
    Let $N$ be a six-dimensional manifold. An \emph{$\mathrm{SU}(3)$-structure} on $N$ is a pair of well defined, nowhere-vanishing forms $(\omega,\Omega)$, where $\omega$ is a real two-form and $\Omega$ is a complex three-form, satisfying the following relations:
\begin{equation*} 
\omega\wedge\Omega = 0\, , \qquad \frac{1}{6}\, \omega\wedge\omega\wedge\omega = \frac{i}{8}
\, \Omega\wedge\overline\Omega \neq 0\, .
\end{equation*} 
We will denote $\Omega_+\coloneqq\mathrm{Re}(\Omega)$ and $\Omega_-\coloneqq\mathrm{Im}(\Omega)$.
\end{definition}

The following lemmas describe how to obtain $\mathrm{G}_2$-structures from $\mathrm{SU}(2)$- and $\mathrm{SU}(3)$-structures.

\begin{lemma}
\label{lem:G2fromSU2}
    Let $X$ be a four-dimensional manifold with $\mathrm{SU}(2)$-structure $\lbrace\omega_i\rbrace_{i=1,2,3}$. Let $Y$ be a three-dimensional manifold and let $\eta^1,\eta^2,\eta^3$ be a global frame of $T^*Y$. Then, the three-form
    \begin{equation*}
    \label{eq:G2formfromSU2}
        \varphi=\eta^1\wedge\eta^2\wedge\eta^3 +\eta^1\wedge\omega_1+\eta^2\wedge\omega_2+\eta^3\wedge\omega_3
    \end{equation*}
    defines a $\mathrm{G}_2$-structure on the product manifold $Y\times X$.
\end{lemma}

\begin{lemma}
\label{lem:G2fromSU3}
    Let $N$ be a six-dimensional manifold with $\mathrm{SU}(3)$-structure $(\omega,\Omega)$. Let $\eta$ be a nowhere-vanishing one-form on the circle $S^1$. Then, the three-form
    \begin{equation*}
    \label{eq:G2formfromSU3}
        \varphi=\eta\wedge\omega+\Omega^+
    \end{equation*}
    defines a $\mathrm{G}_2$-structure on the product manifold $N\times S^1$.
\end{lemma}

\section{Examples of manifolds with \texorpdfstring{$\mathrm{G}_2$}{G2}-structures}
\label{sec:manifolds}

We now present the manifolds and $\mathrm{G}_2$-structures that constitute the main object of study in this note, mostly following \cite{Fino:2023vdp}.

\subsection{\texorpdfstring{$S^3\times T^4$}{S3xT4}}

We first consider the direct product of a three-dimensional sphere and four circles, $S^3\times T^4$, which we identify with the Lie group $\mathrm{SU}(2)\times\mathrm{U}(1)^4$. We choose a basis of the Lie algebra $\mathbb{R}^3\oplus\mathfrak{su}(2)\oplus\mathbb{R}$ such that its dual basis  $\lbrace v^1,\dots,v^7\rbrace$ satisfies the following structure equations
\begin{equation*}
    \dd v^4=v^{56}\,, \qquad \dd v^5=-v^{46}\,, \qquad \dd v^6=v^{45}\,, \qquad \dd v^1=\dd v^2=\dd v^3=\dd v^7=0\,.
\end{equation*}
This manifold can be regarded either as a $T^3$ bundle over the Hopf surface $S^3\times S^1$, or as an $S^3$ bundle over $T^4$. By defining $\mathrm{SU}(2)$ structures on $S^3\times S^1$ and $T^4$ and applying \Cref{lem:G2fromSU2}, we obtain two inequivalent families of integrable $\mathrm{G}_2$-structures over $S^3\times T^4$ with closed torsion three-form. Their torsion classes follow from a straightforward computation:

\begin{proposition}[\cite{Fino:2023vdp}]
\label{prop:G2structureS3T4}
    Consider the manifold $S^3\times T^4$ with a basis of its dual Lie algebra as above. Let $\ell,c_1,c_2,c_3,c_7>0$ be constant positive numbers. The forms
    \begin{align*}
        \omega_1&=\ell\,v^{45}+\sqrt{c_7\ell}\, v^{67} \, , & \omega_2&=\ell\,v^{46}-\sqrt{c_7\ell}\, v^{57} \, , &
        \omega_3&=-\sqrt{c_7\ell}\, v^{47}-\ell\,v^{56} \, , & 
    \end{align*}
    define an $\mathrm{SU}(2)$-structure on the Hopf surface $S^3\times S^1$. Let
    \begin{equation*}
        \eta^1=\sqrt{c_1}\, v^1 \,, \qquad \eta^2=\sqrt{c_2}\, v^2 \,, \qquad \eta^1=\sqrt{c_3}\, v^3 \,, 
    \end{equation*}
    be a global frame of $T^*T^3$, then 
    \begin{equation*}
        \varphi^1=\eta^1\wedge\eta^2\wedge\eta^3 +\eta^1\wedge\omega_1+\eta^2\wedge\omega_2+\eta^3\wedge\omega_3
    \end{equation*}
    determines an integrable $\mathrm{G}_2$-structure on $S^3\times T^4$ with metric
    \begin{equation}
    \label{eq:metricS3T4}
        g=\sum_{i=1}^3 c_i \, v^i\otimes v^i + \ell\sum_{j=4}^6 v^j\otimes v^j + c_7 \, v^7\otimes v^7 \, , 
    \end{equation}
    and torsion classes
    \begin{equation*}
        \tau_0=0 \, , \quad \tau_1=\frac{1}{4}\sqrt{\frac{c_7}{\ell}}v^7 \, , \quad \tau_3=\frac{1}{4}\left( -\eta^{12}\wedge v^4  + \eta^{13}\wedge v^5  + \eta^{23}\wedge v^6-3\ell v^{456} \right) \, .
    \end{equation*}
    The torsion three-form is closed and equal to
    \begin{equation}
    \label{eq:HS3T4}
        H=\ell \, v^{456} \, .
    \end{equation}
\end{proposition}

The following $\mathrm{G}_2$-structure corresponds to one of the examples considered in \cite{Fiset:2021azq} to motivate the $\SV_a$ algebra:

\begin{proposition}
    Consider the manifold $S^3\times T^4$ with a basis of its dual Lie algebra as above. Let $\ell,c_1,c_2,c_3,c_7>0$ be constant positive numbers. The forms
    \begin{align*}
        \tilde{\omega}_1&=\sqrt{c_7 c_1}\,v^{71}+\sqrt{c_2 c_3}\, v^{23} \, , & \tilde{\omega}_2&=\sqrt{c_7 c_2}\,v^{72}-\sqrt{c_1 c_3}\, v^{13} \, , \\
        \tilde{\omega}_3&=-\sqrt{c_7 c_3}\, v^{73}-\sqrt{c_1 c_2}\,v^{12} \, , & &
    \end{align*}
    define an $\mathrm{SU}(2)$-structure on $T^4$. Let
    \begin{equation*}
        \tilde{\eta}^1=\sqrt{\ell}\, v^4 \,, \qquad \tilde{\eta}^2=\sqrt{\ell}\, v^5 \,, \qquad \tilde{\eta}^1=\sqrt{\ell}\, v^6 \,, 
    \end{equation*}
    be a global frame of $T^*S^3$, then 
    \begin{equation*}
        \varphi^2=\tilde{\eta}^1\wedge\tilde{\eta}^2\wedge\tilde{\eta}^3 +\tilde{\eta}^1\wedge\tilde{\omega}_1+\tilde{\eta}^2\wedge\tilde{\omega}_2+\tilde{\eta}^3\wedge\tilde{\omega}_3
    \end{equation*}
    determines an integrable $\mathrm{G}_2$-structure on $S^3\times T^4$ with metric as in \eqref{eq:metricS3T4} and torsion classes
    \begin{equation*}
        \tau_0=\frac{6}{7}\frac{1}{\sqrt{\ell}} \, , \qquad \tau_1=0 \, , \qquad \tau_3=\frac{1}{7\sqrt{\ell}}\varphi^2-\ell v^{456} \, .
    \end{equation*}
    The torsion three-form is closed and given by \eqref{eq:HS3T4}.
\end{proposition}

\begin{remark}
    Note that both $\mathrm{G}_2$-structures in the previous Propositions have the same metric and torsion three-form. This will have interesting implications for our main results regarding vertex algebras in \Cref{sec:embeddings}.
\end{remark}

\subsection{\texorpdfstring{$S^3\times S^3\times S^1$}{S3xS3xS1}}

We now consider the direct product of two three-dimensional spheres and a circle, $S^3\times S^3\times S^1$, which we identify with the Lie group $\mathrm{SU}(2)\times\mathrm{SU}(2)\times\mathrm{U}(1)$. We choose a basis of the Lie algebra $\mathfrak{su}(2)\oplus\mathfrak{su}(2)\oplus\mathbb{R}$ such that its dual basis $\lbrace v^1,\dots,v^7\rbrace$ satisfies the following structure equations
\begin{align*}
    \dd v^1&=v^{23}\,, & \dd v^2&=-v^{13}\,, & \dd v^3&=v^{12}\,, & & \\
    \dd v^4&=v^{56}\,, & \dd v^5&=-v^{46}\,, & \dd v^6&=v^{45}\,, & \dd v^7&=0\,.
\end{align*}
Following \cite{Fino:2023vdp}, we construct three different $\mathrm{G}_2$-structures on $S^3\times S^3\times S^1$ by choosing different $\mathrm{SU}(3)$-structures on $S^3\times S^3$ and applying \Cref{lem:G2fromSU3}.

\begin{lemma}[\cite{Fino:2023vdp}]
\label{lem:SU3structure}
    Consider the manifold $S^3\times S^3$ with a basis of its dual Lie algebra as above. Let $\ell,s>0$ be constant positive numbers. The forms
    \begingroup
    \allowdisplaybreaks
    \begin{align*}
        \omega&=\sqrt{s\ell}\,v^{14} +\sqrt{s\ell}\,v^{25} -\sqrt{s\ell}\,v^{36} \, , \\
        \Omega_+&= \sqrt{s^3}\,v^{123}+\ell\sqrt{s}\,v^{156}-\ell\sqrt{s}\,v^{246}-\ell\sqrt{s}\,v^{345} \, , \\
        \Omega_-&= \sqrt{\ell^3}\,v^{456}+s\sqrt{\ell}\,v^{234}-s\sqrt{\ell}\,v^{135}-s\sqrt{\ell}\,v^{126} \, , 
    \end{align*}%
    \endgroup
    define an $\mathrm{SU}(3)$-structure on $S^3\times S^3$ with holomorphic volume form $\Omega=\Omega_++i\Omega_-$. Similarly, any complex rotation of the holomorphic volume form $\tilde{\Omega}=e^{i\theta}\Omega$ also defines an $\mathrm{SU}(3)$-structure on $S^3\times S^3$.
\end{lemma}

\begin{proposition}
    Consider the manifold $S^3\times S^3\times S^1$ with a basis of the dual Lie algebra as above.  Let $\ell,s,c_7>0$ be constant positive numbers, let $\omega$, $\Omega$ be as in \Cref{lem:SU3structure} and let $\eta=\sqrt{c_7} \, v^7$. Applying \Cref{lem:G2fromSU3} results in the following:
    \begin{itemize}
        \item From $(\omega,\Omega)$ we obtain an integrable $\mathrm{G}_2$-structure that we denote by $\varphi^3$, with torsion classes
            \begin{equation*}
            \tau_0=\frac{6}{7}\frac{1}{\sqrt{s}} \, , \qquad \tau_1=\frac{1}{4}\sqrt{\frac{c_7}{\ell}}v^7 \, , \qquad \tau_3=\frac{1}{7\sqrt{s}}\varphi^3+\frac{1}{4\sqrt{\ell}}\Omega_- -s v^{123}-\ell v^{456} \, .
        \end{equation*}
        \item From $(\omega,\Omega^1\coloneqq e^{i\theta_1}\Omega)$ with $\tan\theta_1=\sqrt{\frac{\ell}{s}}$ we obtain an integrable $\mathrm{G}_2$-structure that we denote by $\varphi^4$, with torsion classes
            \begin{equation*}
            \tau_0=0 \, , \qquad \tau_1=\frac{\sqrt{c_7}}{4}\sqrt{\frac{1}{\ell}+\frac{1}{s}}\,v^7 \, , \qquad \tau_3= \frac{1}{4}\sqrt{\frac{1}{\ell}+\frac{1}{s}}\,\Omega^1_- -s v^{123}-\ell v^{456} \, .
        \end{equation*}
        \item From $(\omega,\Omega^2\coloneqq e^{i\theta_2}\Omega)$ with $\tan\theta_2=-\sqrt{\frac{s}{\ell}}$ we obtain an integrable $\mathrm{G}_2$-structure that we denote by $\varphi^5$, with torsion classes
            \begin{equation*}
            \tau_0=\frac{6}{7}\sqrt{\frac{1}{\ell}+\frac{1}{s}} \, , \qquad \tau_1=0 \, , \qquad \tau_3= \frac{1}{7}\sqrt{\frac{1}{\ell}+\frac{1}{s}}\varphi^5 -s v^{123}-\ell v^{456} \, .
        \end{equation*}
    \end{itemize}
    All three $\mathrm{G}_2$-structures have metric
    \begin{equation*}
        g=s\sum_{i=1}^3 v^i\otimes v^i + \ell\sum_{j=4}^6 v^j\otimes v^j + c_7 \, v^7\otimes v^7 \, , 
    \end{equation*}
    and closed torsion three-form equal to
    \begin{equation*}
        H=s \, v^{123}+ \ell \, v^{456} \, .
    \end{equation*}
\end{proposition}

\section{Superaffine algebras and the deformed Shatashvili--Vafa}

\subsection{SUSY vertex algebras}

In this section we review the basics on SUSY vertex algebras and present those that play a role in this note, including the deformed Shatashvili--Vafa vertex algebra $\SV_a$. For a detailed introduction to vertex algebras, we refer the reader to \cite{Kac:1996wd}. We will use extensively the formalism introduced by Heluani--Kac \cite{Heluani:2006pk}, although the $\mathcal{N}=1$ case we employ here was studied earlier by Barron \cite{Barron:1999ye}.

A \emph{vertex algebra} is a vector superspace $V$ (that is, a vector space with a $\mathbb{Z}_2$-grading) endowed with a non-vanishing even vector $\vert 0 \rangle\in V$ (vacuum), an even endomorphism $T:V\rightarrow V$ (infinitesimal translation) and a parity-preserving linear map $Y:V\rightarrow \mathrm{End}(V)[[z^\pm]]$ (state-field correspondence) mapping each vector to a \emph{field}, defined as a formal sum
    \begin{equation*}
        a(z)\coloneqq Y(a,z)=\sum_{n\in\mathbb{Z}}a_{(n)}z^{-n-1} \, ,
    \end{equation*}
in an even variable $z$, with \emph{Fourier modes} $a_{(n)} \in \mathrm{End}(V)$. This data must satisfy that $Y (a, z)b$ is a formal Laurent series for all $b \in V$, so that the operator product expansions are finite sums
    \begin{equation*}
        a(z)b(w)\sim \sum_{n\geq 0}\frac{(a_{(n)}b)(w)}{(z-w)^{n+1}}\, ,
    \end{equation*}
along with the vacuum axioms, translation invariance and locality (see e.g. \cite{Kac:1996wd}). We will call $a_{(n)}b$ the \emph{$n$-product} of $a$ and $b$.

When the vertex algebra is \emph{conformal}, the $T$-action is enhanced to an action of the Virasoro algebra for some central charge $c\in\mathbb{C}$. We are interested in the case where this action is further enhanced to an $\mathcal{N}=1$ superconformal symmetry by including an additional odd linear map $S:V\rightarrow V$ such that $S^2=T$. This motivates the definition of a SUSY vertex algebra \cite{Heluani:2006pk}.

A \emph{SUSY vertex algebra} is a tuple $(V,\vert0\rangle,S,Y(\cdot,z))$, where $(V,\vert0\rangle,T = S^2,Y(\cdot,z))$ is a vertex algebra and $S:V\rightarrow V$ is an odd linear map which is furthermore a derivation for the $n$-products.

We can now define the \emph{$\lambda$-bracket} (the formal Fourier transform by an even formal parameter $\lambda$ of the operator product expansion) and the \emph{normally ordered product} as follows
\begin{equation*}
    [a_\lambda b]=\sum_{n\in\mathbb{Z}} \frac{\lambda^{n}}{n!}a_{(n)}b\, , \qquad \normord{ab}\,=a_{(-1)}b \, .
\end{equation*}
Just as a Lie algebra freely generates the universal enveloping algebra, a $\mathbb{C}[T]$-module, $\lambda$-bracket satisfying a certain set of axioms---that can be found in \cite{Heluani:2006pk}---and an odd derivation $S$ with respect to the $\lambda$-bracket such that $S^2=T$ freely generate a universal enveloping SUSY vertex algebra. The following example is constructed in this way:

\begin{example}\label{example:superfaffine}
    Let $(\mathfrak{g},(\cdot\vert\cdot))$ be a quadratic Lie 
   algebra and $k\in\mathbb{C}$ a scalar. Let $\Pi\mathfrak{g}$ be the corresponding purely odd vector superspace, identifying elements $a$ of $\mathfrak{g}$ with their corresponding odd copies $\Pi a$ in $\Pi\mathfrak{g}$. The \emph{universal superaffine vertex algebra} with level $k$ associated to $\mathfrak{g}$ is the SUSY vertex algebra $V^k(\mathfrak{g}_{\text{super}})$ freely generated by the $\mathbb{C}[T]$-module $\mathfrak{g}\oplus\Pi\mathfrak{g}$ (and the scalar $k$), the odd derivation $S$ defined by
    \begin{equation*}
        Sa=T\,\Pi a \, , \qquad S \, \Pi a = a \, ,
    \end{equation*}
    for all $a\in\mathfrak{g}$, and the $\lambda$-brackets
    \begin{equation*}
        [a_\lambda b]=[a,b]+\lambda (a\vert b)k \, , \qquad  [a_\lambda \Pi b]=\Pi [a,b] \, , \qquad [\Pi a_\lambda \Pi b]=(b\vert a)k \, ,
    \end{equation*}
    for all $a,b\in\mathfrak{g}$, which satisfy the axioms of \cite{Heluani:2006pk}.
\end{example}

\begin{remark}
For the quadratic Lie algebras we will consider, the universal superaffine vertex algebra with level $k=2$ embeds into the chiral de Rham complex of the corresponding Lie group. For more details, see \Cref{sec:embeddingsCDR}.
\end{remark}

\subsection{The deformed Shatashvili--Vafa vertex algebra}\label{sec:deformedSV}

The \emph{deformed Shatashvili--Vafa algebra} with parameter $a\in\mathbb{C}$, which we will denote by $\SV_a$, is the SUSY vertex algebra freely generated by the $\mathbb{C}[T]$-module
    \begin{equation*}
        \mathbb{C}G\oplus \mathbb{C}L\oplus \mathbb{C}\Phi\oplus \mathbb{C}K\oplus \mathbb{C}X\oplus \mathbb{C}M \, ,
    \end{equation*}
    the odd derivation $S$ defined by
    \begin{equation*}
        S G=2L \, , \quad 2S L = T\,G \, , \quad S \Phi=K \, , \quad S K = T\,\Phi \, , \quad S X=M \, , \quad S M = T\,X \, ,
    \end{equation*} 
    and the $\lambda$-brackets
    \begingroup
    \allowdisplaybreaks
    \begin{align*}
        [L_\lambda L]&= (T +2\lambda) L + \left( \frac{21}{2}+3a^2 \right)\frac{\lambda^3}{12} \, , \qquad
        [L_\lambda G]= \left(T +\frac{3}{2}\lambda\right) G \, , \\
        [G_\lambda G]&= 2 L + \left( \frac{21}{2}+3a^2 \right)\frac{\lambda^2}{3} \, , \qquad
        [L_\lambda \Phi]= \left(T +\frac{3}{2}\lambda\right) \Phi \, , \\   
        [L_\lambda X]&= -\frac{7}{24}\lambda^3+ (T +2\lambda) X  \, , \qquad  
        [G_\lambda \Phi]= a\,\frac{\lambda^2}{2}+ K \, , \\   
        [G_\lambda X]&= \left(-\frac{1}{2}G+a\Phi\right)\lambda+M  \, , \qquad 
        [\Phi_\lambda \Phi]= -\frac{7}{2}\lambda^2+6X \, , \\  
        [\Phi_\lambda X]&= -\left(\frac{5}{2}T + \frac{15}{2}\lambda \right) \Phi  \, , \qquad
        [X_\lambda X]= \frac{35}{24}\lambda^3 -\left(5T+10\lambda \right) X  \, , \\
        [L_\lambda K]&= a\frac{\lambda^3}{4} + \left(T +2\lambda\right) K \, , \qquad
        [L_\lambda M]= \left(-\frac{1}{2}G+a\Phi\right)\frac{\lambda^2}{2}+ \left(T +\frac{5}{2}\lambda\right) M  \, , \\  
        [G_\lambda K]&= (T+3\lambda)\Phi \, , \qquad  
        [G_\lambda M]= -\frac{7}{12}\lambda^3+\left(L-a K\right)\lambda+ (T+4\lambda)X  \, , \\   
        [\Phi_\lambda K]&= a(3T+6\lambda)\Phi-\left( \frac{3}{2}T +3\lambda \right)G -3M \, , \\  
        [\Phi_\lambda M]&= -a\frac{5}{4}\lambda^3+a(3T+6\lambda)X +\left(-\frac{1}{2}T + \frac{9}{2}\lambda \right) K +3\normord{\Phi G} \, , \\
        [X_\lambda K]&= -a\frac{5}{4}\lambda^3+a(3T+6\lambda)X -3K\lambda -3\normord{\Phi G}  \, , \\
        [X_\lambda M]&=  +a\frac{4}{7}\normord{X\Phi} +a\left( -\frac{27}{14}T^2-3T\lambda-3\lambda^2 \right)\Phi \\
        &+\left(\frac{1}{4}T^2 -\frac{9}{4}T\lambda -\frac{9}{4}\lambda^2\right)G +\left( \frac{1}{2}T-5\lambda \right) +4\normord{XG}  \, , \\   
        [K_\lambda K]&= -\frac{7}{2}\lambda^3+ a(3T+6\lambda)K +(3T+6\lambda)X -(3T+6\lambda)L \, , \\  
        [K_\lambda M]&= a(3T+6\lambda)M -\left(\frac{11}{2}T\lambda +\frac{15}{2}\lambda^2\right)\Phi +3\normord{GK}-6\normord{L\Phi} \, , \\  
        [M_\lambda M]&= -\frac{35}{24}\lambda^4 - a(3T\lambda+3\lambda^2)K  + a(\frac{11}{5}T^2+10T\lambda+10\lambda^2)X \\
        &- \left(\frac{5}{2}T^2 -\frac{9}{2}T\lambda -\frac{9}{2}\lambda^2\right) -6\normord{GM} +12\normord{LX} +\frac{2}{5}\normord{XX} -\normord{KK} \, ,
    \end{align*}%
    \endgroup
    which satisfy all the required axioms \cite{Heluani:2006pk} provided that the following condition is imposed:
    \begin{equation}
    \label{eq:singularvector}
        0=a\left( \frac{9}{7}T^2\Phi-\frac{8}{7}\normord{\Phi X} \right)+4\normord{GX}-2\normord{\Phi K}-4T\,M-T^2G \, .
\end{equation}
In particular, the vectors $G$ and $L$ in $\SV_a$ define a Neveu-Schwarz vertex algebra with central charge $c=\frac{21}{2}+3a^2$.

\begin{remark}\label{rem:SW}
In the case $a = 0$, the relation \eqref{eq:singularvector} was more conceptually understood in \cite{Heluani:2014uaa}, where the Shatashvili--Vafa vertex algebra $\SV_0$ was obtained as a quotient of a specific universal superaffine vertex algebra (cf. Example \ref{example:superfaffine}). In this setup, one first undertakes quantum Hamiltonian reduction to obtain a two-parameter $W$-algebra $\mathcal{SW}(\tfrac{3}{2},\tfrac{3}{2},2)$, and then quotient by an ideal generated by $\eqref{eq:singularvector}$, for $a = 0$ and a suitable choice of parameters, to obtain $\SV_0$. It was then observed in \cite{Fiset:2021azq} that a similar procedure applies to $\SV_a$, with the more general relation \eqref{eq:singularvector}.
\end{remark}

\begin{remark}
Note that the parameter $a$ corresponds to $i\sqrt{\tfrac{2}{\mathrm{k}}}$ in \cite{Fiset:2021azq}, where $\mathrm{k}$ is related to the radius of the $\mathrm{AdS}_3$ spacetime via 
$$
\mathrm{k}=R_{\mathrm{AdS}_3}^2/\alpha'.
$$ 
In the limit $a\rightarrow0$ we formally recover a Minkowski spacetime and the algebra reduces to that of Shatashvili--Vafa \cite{Shatashvili:1994zw}.
\end{remark}

\section{Embeddings}
\label{sec:embeddings}

\subsection{Chiral de Rham complex}
\label{sec:embeddingsCDR}

We start recalling the coordinate independent description of the chiral de
Rham complex using Courant algebroids \cite{Heluani:2008susyII,Heluani:2008hw}, which we shall apply later to the situation of a Lie group. We refer the reader to \cite{Alvarez-Consul:2020hbl} and references therein for details.

Let $E$ be an exact Courant algebroid over a smooth manifold $M$, with exterior differential $d$, given by a short exact sequence of smooth bundles of the form
\begin{equation}\label{eq:exactE}
0 \longrightarrow T^*M \overset{\pi^*}{\longrightarrow} E \overset{\pi}{\longrightarrow} TM \longrightarrow 0 \, .
\end{equation}
We will denote the indefinite pairing and the Dorfman bracket on sections of $E$ by $\langle , \rangle$ and $[,]$, respectively. Recall that a choice of isotropic splitting $s \colon T \to E$ of \eqref{eq:exactE} induces a vector bundle isomorphism $E \cong T \oplus T^*$, such that  
\begin{equation*}\label{eq:pairingbracketexp}
\langle X + \xi, X + \xi \rangle = \xi(X) \, , \qquad [X+ \xi,Y + \eta] = [X,Y] + L_X \eta - i_Y d\xi + i_Yi_X H \, ,
\end{equation*}
where $H \in \Omega^3(M)$ is a closed three-form determined by $s$.

Let $\Pi E$ be the corresponding purely odd super vector bundle. We denote by $\underline{\mathbb{C}}$ the sheaf of locally constant functions on $M$. The next result provides a coordinate-free description of the chiral de Rham complex on the smooth manifold $M$.

\begin{proposition}[\cite{Heluani:2008susyII,Heluani:2008hw}]\label{prop:CDRE}
There exists a unique sheaf of SUSY vertex algebras $\Omega^\mathrm{ch}_M(E)$ over $M$ endowed with embeddings of sheaves of $\underline{\mathbb{C}}$-modules
$$
\iota \colon C^\infty(M) \hookrightarrow \Omega^\mathrm{ch}_M(E) \, , \qquad j \colon E \oplus \Pi E \hookrightarrow \Omega^\mathrm{ch}_M(E) \, ,
$$
satisfying the following properties:

\begin{enumerate}

\item $\iota$ is an isomorphism of unital commutative algebras onto its image
$$
\iota(fg) = \, \normord{\iota(f)\iota(g)} \, ,
$$

\item $\iota$ and $j$ are compatible with the $C^\infty(M)$-module structure of $E \oplus \Pi E$ and the odd endomorphism $S$ of $\Omega^\mathrm{ch}_M(E)$
$$
[\iota(f)_\lambda \iota(g)] = 0 \, , \quad j(fA) =\, \normord{\iota(f)j(A)} \, , \quad S\iota(f) = j(\Pi df) \, ,
$$

\item $\iota$ and $j$ are compatible with the Dorfman bracket and pairing
$$
{[j(A)}_\lambda j(B) ] = j([A,B]) + 2 \lambda \iota (\langle A,B\rangle) \, , \quad {[j(A)}_\lambda j(\Pi B)] = j(\Pi[A,B]) \, ,$$
$${[j(\Pi A)}_\lambda j(\Pi B)] = 2 \iota (\langle A,B\rangle) \, ,
$$

\item $\iota$ and $j$ are compatible with the action of $\Gamma(E \oplus \Pi E)$ on $C^\infty(M)$
$$
{[j(A)}_\lambda \iota(f)] = \iota(\pi(A)(f)) \, , \quad {[j(\Pi A)}_\lambda \iota(f)] = 0 \, ,
$$

\item  $\Omega^\mathrm{ch}_M(E)$ is universal with these properties,

\end{enumerate}
for all $f,g \in C^\infty(M)$, $A,B \in \Gamma(E)$.
\end{proposition}

To apply this to our results in the next section, assume now that our manifold $M$ is a compact Lie group $K$, and let $E$ be a left-equivariant exact Courant algebroid over $K$. By \cite[Proposition 4.5]{Alvarez-Consul:2020hbl}, we can associate to $E$ a quadratic Lie algebra given by the invariant sections of $E$
$$
\mathfrak{g} = \Gamma(E)^K \, ,
$$
endowed with the induced bracket and pairing. Applying the universal construction in Proposition \ref{prop:CDRE}, we obtain an embedding of the superaffine vertex algebra $V^2(\mathfrak{g}_{\text{super}})$ of level $k =2$ (see \cite[Proposition 4.17]{Alvarez-Consul:2020hbl}).

\begin{proposition}[\cite{Alvarez-Consul:2020hbl}]\label{prop:superaffineembed}
Let $K$ be a compact Lie group and $E$ a left-equivariant exact Courant algebroid over $K$. Then there is an embedding
$$
V^2(\mathfrak{g}_{\text{super}}) \hookrightarrow \Gamma(K,\Omega^\mathrm{ch}_K(E))
$$
of the superaffine vertex algebra $V^2(\mathfrak{g}_{\text{super}})$ of level $k =2$ on the space of global sections $\Gamma(K,\Omega^\mathrm{ch}_K(E))$ of $\Omega^\mathrm{ch}_K(E)$.
\end{proposition}

\subsection{Embeddings of \texorpdfstring{$\SV_a$}{SVa}}
\label{sec:embeddingsSVa}

For the remainder of the section, we will denote by $M$ either the manifold $S^3\times T^4$ or $S^3\times S^3\times S^1$ and by $\mathfrak{k}$ the corresponding Lie algebra. Let $\lbrace v_i \rbrace$ be a basis of $\mathfrak{k}$ with dual basis $\lbrace v^i\rbrace$. We will denote by $\varphi$ any of the $\mathrm{G}_2$-structures introduced in \Cref{sec:manifolds}, that is: $\varphi^i$ with $i=1,2$ for $S^3\times T^4$ and $i=3,4,5$ for $S^3\times S^3\times S^1$. By direct application of \cite[Proposition 4.5]{Alvarez-Consul:2020hbl} we obtain the following.

\begin{proposition} \label{prop:quadraticLiealgebra}
    Let $M$, $\varphi$ and $\mathfrak{k}$ be as above, and let $H$ be the torsion three-form of $\varphi$ introduced in \Cref{prop:uniqueG2connection}. The vector space $\mathfrak{g}=\mathfrak{k}\oplus\mathfrak{k}^*$ endowed with the bracket
    \begin{equation*}
        [v+\alpha,w+\beta]=[v,w]-\beta([v,])+\alpha([w,])+i_w i_v H \, ,
    \end{equation*}
and pairing
    \begin{equation*}
        (v+\alpha \, \vert \, v+\alpha)=v(\alpha) \, ,
    \end{equation*}
for any $v,w\in\mathfrak{k}$, $\alpha,\beta\in\mathfrak{k}^*$, has the structure of a quadratic Lie algebra.
\end{proposition}

For each choice of $M$ and $\varphi$ we now present explicit embeddings of the $\SV_a$ algebra in the corresponding superaffine vertex algebra $V^k(\mathfrak{g}_{\text{super}})$, for arbitrary level $0\neq k\in\mathbb{C}$. When $k=2$, applying Proposition \ref{prop:superaffineembed}, we obtain induced embeddings in the global sections of the chiral de Rham complex
$$
\SV_a \hookrightarrow  V^2(\mathfrak{g}_{\text{super}}) \hookrightarrow \Gamma(M,\Omega^\mathrm{ch}_M(E))\,.
$$
To the best of our knowledge, these constitute the first explicit embeddings of the $\SV_0$ algebra when the underlying $\mathrm{G}_2$-structure has torsion, and the first ever explicit embeddings of $\SV_a$ for $a\neq 0$.

Following \cite{RodriguezDiaz:2016tih}, we define the following odd vectors of $V^k(\mathfrak{g}_{\text{super}})$
\begin{equation*}
    e^i=\frac{1}{\sqrt{2}}\Pi\left( g^{ij} v_j+v^i \right) \, .
\end{equation*} 

\begin{theorem}
\label{thm:finalresult}
    Let $M$, $\varphi$ be as above, let $\mathfrak{g}$ be as in \Cref{prop:quadraticLiealgebra} and let $0\neq k\in\mathbb{C}$. 
    Then, the following sections of  $V^k(\mathfrak{g}_{\text{super}})$
    \begin{equation*}
        \Phi=\frac{1}{3k}\sqrt{\frac{2}{k}}\, \varphi_{ijk}\normord{e^i\normord{e^je^k}} \, , \qquad K=S\Phi\,, 
    \end{equation*}
    generate an embedding of the $\SV_a$ algebra in $V^k(\mathfrak{g}_{\text{super}})$ for the choice of parameter
    \begin{equation*}
        a=-\frac{1}{\sqrt{k}}\frac{7}{6}\tau_0 \, .
    \end{equation*}
    The remaining fields are easily recovered via
    \begin{equation*}
        X=\frac{1}{6}\Phi_{(0)}\Phi\,,  \qquad M=S X\,, \qquad G=-\frac{1}{3}\Phi_{(1)}K -2\frac{1}{\sqrt{k}}\frac{7}{6}\tau_0\,\Phi\,, \qquad L=\frac{1}{2}S G\,.
    \end{equation*}
\end{theorem}

\begin{proof}
    The proof is a straightforward but long verification that the sections defined above satisfy the $\lambda$-brackets introduced in Section \ref{sec:deformedSV}. The condition \eqref{eq:singularvector} is directly satisfied for this embedding. The computation is more easily performed with the assistance of computer software: we have used the \emph{OPEdefs} \emph{Mathematica} package by Thielemans \cite{Thielemans:1994er}.
\end{proof}

\begin{remark}
Our results shall be compared with the embedding of $\SV_0$ in the chiral de Rham complex of a $\mathrm{G}_2$-holonomy manifold found by Rodr\'iguez D\'iaz \cite{RodriguezDiaz:2016tih}, which uses a local expression for the generator $\Phi$ in normal coordinates with an explicit dependence of the Christoffel symbols of the Levi-Civita connection. The simpler expression for $\Phi$ in our main result is due to the fact that the connection with skew torsion introduced in \Cref{prop:uniqueG2connection} has vanishing Christoffel symbols in the given parallel frames.
\end{remark}

\begin{remark}
\label{rem:tau1enG}
The embedding constructed in Theorem \ref{thm:finalresult} is equivalently determined by the sections $\Phi$ and $G$, where the latter can be defined independently of $K$ as
    \begin{equation*}
        G=G_0+\frac{4\sqrt{2}}{k} T(\tau_{1\,i}\,e^i) \, ,
    \end{equation*}
    where
    \begin{equation*}
        G_0=\frac{1}{k}\left( 2 g_{ij} \normord{(Se^i) e^j}+\frac{1}{3k} 4 
g_{ij}g_{mn}\normord{e^i\normord{e^m \left[e^n,e^j \right] }} \right) \,.
\end{equation*}
Note that the expression for $G_0$ only depends on the metric $g$ and the closed torsion three-form $H$, while $G$ receives a \emph{dilaton correction} from the Lee form $\tau_1$, as found in \cite{Alvarez-Consul:2020hbl,Alvarez-Consul:2023zon}. By \cite[Theorem 4.9]{Silva:2024fvl}, this fulfills our general expectation that a generalized Ricci flat metric with closed divergence should induce an embedding of the $\mathcal{N}=1$ vertex algebra in the global sections of the chiral de Rham complex. Similarly as above, the generator of supersymmetry $G$ should receive a dilaton correction proportional to $T$ acting on the divergence. Note also that $G_0$ is formally similar to the $\mathcal{N}=1$ generator in the Kac--Todorov construction \cite{Heluani:2006pk}, although our formula differs from the standard one, since we are not using a pair of dual basis.
\end{remark}

\end{document}